\documentclass[a4paper,12pt]{article}
\usepackage{amsfonts,amsmath,amssymb,amsthm,verbatim,placeins,color,comment,placeins,caption,subcaption,enumerate,bm}
\usepackage{graphicx}
\usepackage{geometry}
\usepackage[pdfauthor={DG},pdfstartview={FitH},colorlinks=true,citecolor=blue]{hyperref}
\usepackage[english]{babel}
\usepackage[sort&compress]{natbib}

\newtheorem{theorem}{Theorem}[section]

\newtheorem{lemma}{Lemma}[section]

\theoremstyle{remark}

\theoremstyle{definition}
\newtheorem{defi}{Definition}[section]
\newtheorem{example}{Example}[section]

\theoremstyle{remark}
\newtheoremstyle{myremark}{}{}{\color{blue}\small}{}{\color{blue}\bfseries}{}{ }{}

\theoremstyle{myremark}

\newcommand{\E}{\mathbb{E}} %for expectation
\newcommand{\corr}{\operatorname{Corr}} %for correlation
\newcommand{\R}{{\mathbb R}}
\newcommand{\N}{{\mathbb N}}
\newcommand{\Z}{{\mathbb Z}}

\makeatletter
\renewcommand*{\@fnsymbol}[1]{\ensuremath{\ifcase#1\or *\or \mathparagraph\or \ddagger\or
        \mathsection\or \mathparagraph\or \|\or **\or \dagger\dagger
        \or \ddagger\ddagger \else\@ctrerr\fi}}
\makeatother

\begin{document}

\renewcommand*{\thefootnote}{\fnsymbol{footnote}}

\begin{center}
\Large{\textbf{Intermittency of trawl processes}}\\
\bigskip
\large{\today}\\
\bigskip
Danijel Grahovac$^1$\footnote{dgrahova@mathos.hr}, Nikolai N.~Leonenko$^2$\footnote{LeonenkoN@cardiff.ac.uk}, Murad S.~Taqqu$^3$\footnote{murad@bu.edu}\\
\end{center}

\bigskip
\begin{flushleft}
\footnotesize{
$^1$ Department of Mathematics, University of Osijek, Trg Ljudevita Gaja 6, 31000 Osijek, Croatia\\
$^2$ School of Mathematics, Cardiff University, Senghennydd Road, Cardiff, Wales, UK, CF24 4AG}\\
$^3$ Department of Mathematics and Statistics, Boston University, Boston, MA 02215, USA
\end{flushleft}

\bigskip

\textbf{Abstract: }
We study the limiting behavior of continuous time trawl processes which are defined using an infinitely divisible random measure of a time dependent set. In this way one is able to define separately the marginal distribution and the dependence structure. One can have long-range dependence or short-range dependence by choosing the time set accordingly. We introduce the scaling function of the integrated process and show that its behavior displays intermittency, a phenomenon associated with an unusual behavior of moments.

\bigskip

\section{Introduction}

Trawl processes form a class of stationary infinitely divisible processes that allow the marginal distribution and dependence structure to be modelled independently from each other (see \cite{barndorff2011stationary}, \cite{barndorff2014integer} and \cite{barndorff2015recent}). They are defined by
\begin{equation}\label{cttrawlp}
    X(t) = \Lambda (A_t), \quad t \in \R,
\end{equation}
where $\Lambda$ is a homogeneous infinitely divisible independently scattered random measure (\emph{L\'evy basis}) and $A_t = A + (0, t)$ for some Borel subset $A$ of $\R\times \R$ of finite Lebesgue measure. The set $A$ is called the \emph{trawl} and is usually specified using the \emph{trawl function} $g : [0,\infty) \to [0,\infty)$ as
\begin{equation*}
  A = \left\{ (\xi, s) : 0 \leq \xi \leq g(-s), \ s \leq 0 \right\},
\end{equation*}
so that
\begin{equation*}
  A_t = \left\{ (\xi, s) : 0 \leq \xi \leq g(t-s), \ s \leq t \right\}.
\end{equation*}
As explained in \cite{barndorff2015recent} the trawl $A$ can be regarded as a fishing net dragged along the sea, so that at time $t$ it is in position $A_t$. A similar structure can be found in \cite{wolpert2005fractional}. To any L\'evy basis $\Lambda$ there corresponds a L\'evy process $L=\{L(t), \, t \geq 0\}$ referred to as the \emph{L\'evy seed}. The choice of the L\'evy seed determines the marginal law of the trawl process, while the shape of the trawl set $A$ controls the dependence structure. In particular, taking the trawl function to be $-(\alpha+1)$-regularly varying at infinity for some $\alpha \in (0,1)$, one obtains long-range dependence of the resulting trawl process. See Section \ref{sec2} for details.

A discrete time analog of the trawl process \eqref{cttrawlp} has been defined in \cite{doukhan2016discrete} as a process
\begin{equation}\label{discretetrawl}
  Y(k) = \sum_{j=0}^\infty Z^{(k-j)} (a_j), \quad k \in \Z,
\end{equation}
where $Z^{(k)}=\{Z^{(k)}(u), \, u \in \R\}$, $k\in \Z$ are i.i.d.~copies of some process $Z=\{Z(u), \, u \in \R\}$ stochastically continuous at zero and $(a_j)_{j\in\N}$ is a sequence of constants such that $a_j \to 0$ as $j \to \infty$. The long-range dependent case in the discrete time setting corresponds to choosing a sequence $a_j=L(j) j^{-\alpha-1}$ where $L$ is some slowly varying function.

The correspondence of $Y(k)$ in \eqref{discretetrawl} with the continuous time trawl process \eqref{cttrawlp} is the following. Suppose on one hand that $\{Y_k, \, k\in \Z\}$ is a discrete time trawl process with trawl sequence $(a_j)_{j\in\N_0}$ and such that $Z$ is some two-sided L\'evy process $L=\{L(t), \, t \in \R\}$. On the other hand, let $\{X(t), \, t \in \R\}$ be a trawl process with L\'evy seed process $L$ and trawl specified by the function
\begin{equation*}
  g(x) = \sum_{j=0}^\infty a_j \bm{1}_{(-j-1, -j]}(x).
\end{equation*}
Then $\{Y(k), \, k\in \Z\}$ is equal in law to a discretized process $\{X(k), \, k \in \Z\}$ (\cite{doukhan2016discrete}). While the marginal distribution of the trawl process $X(t)$ in \eqref{cttrawlp} is necessarily infinitely divisible, the discrete time setting allows for rather general seed processes.

An important and interesting question regarding trawl processes are limit theorems for cumulative processes arising from them. Assuming the trawl process has zero mean, in the discrete time setup, the cumulative process would be a partial sum process $S_n(t)=\sum_{k=1}^{[nt]} Y(k)$ while in the continuous time it is natural to consider the integrated process $X^*(t) = \int_0^t X(u) du$. However, as we show in this paper, the corresponding limiting behavior of moments seems to be unexpected.

\cite{doukhan2016discrete} have interesting results. In their paper, a limit theorem is proved with convergence to fractional Brownian motion for the partial sum process formed from a zero mean long-range dependent discrete time trawl process \cite[Theorem 1.(i)]{doukhan2016discrete}. The crucial condition for this result is the following small time moment asymptotics of the seed process: for some $\delta>0$, one has
\begin{equation}\label{e:doukhcon}
	\E |Z(t)|^{2+\delta} = O(|t|^{\frac{2+\delta}{2}}), \quad \text{ as } t \to 0.
\end{equation}
One may wonder whether the proof of \cite[Theorem 1.(i)]{doukhan2016discrete} could be extended to the continuous time trawl processes. The following argument shows that the condition \eqref{e:doukhcon} excludes the possibility that the seed process is any L\'evy process except Brownian motion. Indeed, suppose $Z$ is a L\'evy process with L\'evy measure $\nu$ such that $\E Z(1)=0$. By \cite[Lemma 3.1]{asmussen2001approximations} for any $\delta\geq 0$ such that $\E |Z(1)|^{2+\delta} < \infty$, one has
\begin{equation*}
  \lim_{n \to \infty} n \E |Z(1/n)|^{2+\delta} = \int_{\R} |x|^{2+\delta} \nu(dx).
\end{equation*}
Hence, $\E |Z(t)|^{2+\delta} \sim C_\delta  t$ as $t\to 0$ for any $\delta>0$ and \eqref{e:doukhcon} cannot hold unless $\nu=0$ and $Z$ is a Brownian motion. Since Brownian motion is self-similar with self-similarity parameter $1/2$, condition \eqref{e:doukhcon} holds for Brownian motion but not for any other L\'evy process. Hence, the conditions of \cite[Theorem 1(i)]{doukhan2016discrete} cannot be adapted to obtain a limit theorem for a continuous time trawl process \eqref{cttrawlp} when generated by a non-Gaussian seed process.

Our focus in this paper is on the convergence of moments. We prove that the integrated long-range dependent non-Gaussian trawl processes satisfying certain regularity assumptions on the trawl, have a specific limiting behavior called \emph{intermittency}. A precise definition is given in Section \ref{sec3}. Such a property has so far been established for a partial sum and integrated process of superpositions of Ornstein-Uhlenbeck type processes (see \cite{GLST2015} and \cite{GLST2016}). This result sheds a new light on the limiting behaviour related to trawl processes.

\section{Trawl processes}\label{sec2}

In this section we define trawl processes following \cite{barndorff2011stationary}, \cite{barndorff2014integer} and \cite{barndorff2015recent}.

\subsection{Preliminaries}
Let
\begin{equation*}
\kappa_Y(\zeta)=C\left\{ \zeta \ddagger Y\right\} = \log \E e^{i \zeta Y}
\end{equation*}
denote the cumulant (generating) function of a random variable $Y$ and, assuming it exists, $\kappa_Y^{(m)}$ for $m \in \N$ will denote the $m$-th cumulant of $Y$, that is
\begin{equation*}
  \kappa_Y^{(m)} = (-i)^m \frac{d^m}{d\zeta^m} \kappa_Y(\zeta) \big|_{\zeta=0}.
\end{equation*}
If $\kappa_Y(\cdot)$ is analytic around the origin, then
\begin{equation}\label{e:cgf}
\kappa_Y(\zeta)=\sum _{m=1}^\infty \frac {(i\zeta )^m}{m!}\kappa_Y^{(m)}.
\end{equation}
For a stochastic process $Y=\{Y(t)\}$ we write $\kappa_Y(\zeta,t) = \kappa_{Y(t)}(\zeta)$, and by suppressing $t$ we mean $\kappa_Y(\zeta)=\kappa_Y(\zeta,1)$, that is the cumulant function of the random variable $Y(1)$. Similarly, for the cumulants of $Y(t)$, we use the notation $\kappa_Y^{(m)}(t)$, and $\kappa_Y^{(m)}$ for $\kappa_Y^{(m)}(1)$. Recall that the cumulant function of infinitely divisible random variable $Y$ has the L\'{e}vy-Khintchine representation
\begin{equation}\label{e:kappaL}
\kappa(\zeta) = C\left\{ \zeta \ddagger Y\right\} =ia\zeta -\frac{b}{2}\zeta^{2}+\int_{\R}\left( e^{i\zeta x}-1-i\zeta \mathbf{1}_{[-1,1]}(x)\right) \nu(dx), \quad \zeta \in \R
\end{equation}
where $a\in \R$, $b>0$, and the \textit{L\'{e}vy measure} $\nu$ is a deterministic Radon measure on $\R\backslash \{0\}$ such that $\nu\left( \left\{ 0\right\} \right) =0$ and $\int_{\R} \min \left\{ 1,x^{2}\right\} \nu(dx)<\infty$. The triplet $(a,b,\nu)$ is referred to as the \textit{characteristic triplet}. A stochastic process $\{L(t), \, t\geq 0\}$ with stationary, independent increments and continuous in probability ($L(t) \to^P 0$ as $t\to 0$) has a c\`adl\`ag modification which we refer to as a \textit{L\'evy process}. For any infinitely divisible random variable $Y$, there is a corresponding L\'evy process $\{L(t), \, t \geq 0\}$ such that $Y =^d L(1)$.

Next, we review some basic facts about (homogeneous) \textit{L\'evy bases} on $\R^d$, $d \in \N$. A L\'evy basis on $\R^d$ is an infinitely divisible independently scattered random measure, that is, a collection of random variables $\Lambda=\left\{ \Lambda(A), \, A\in \mathcal{B}_b(\R^d) \right\}$ where $\mathcal{B}_b(\R^d)$ denotes the family Borel subsets of $\R^d$ with finite Lebesgue measure. That $\Lambda$ is independently scattered random measure means that for every sequence $\left\{ A_{n}\right\} $ of disjoint sets in $\mathcal{B}_b(\R^d)$, the random variables $\Lambda(A_{n})$, $n=1,2,...$ are independent and
\begin{equation*}
\Lambda\left( \bigcup\limits_{n=1}^{\infty }A_{n}\right) =\sum_{n=1}^{\infty} \Lambda(A_{n}) \quad a.s.
\end{equation*}%
whenever $\bigcup_{n=1}^{\infty}A_{n}\in \mathcal{B}_b(\R^d)$. Moreover, $\Lambda$ is infinitely divisible in the sense that for any collection of sets $A_1, \dots, A_n \in \mathcal{B}_b(\R^d)$ the random vector $\left(\Lambda(A_1), \dots, \Lambda(A_n) \right)$ is infinitely divisible. We will be dealing only with \textit{homogeneous} L\'evy bases which have the property that for every $A \in \mathcal{B}_b(\R^d)$ the cumulant function of $\Lambda(A)$ is given by
\begin{equation*}
  C \left\{ \zeta \ddagger \Lambda(A) \right\} = Leb (A) \kappa (\zeta)
\end{equation*}
where $Leb$ denotes the Lebesgue measure and $\kappa$ is the cumulant function of some infinitely divisible law having the L\'{e}vy-Khintchine representation \eqref{e:kappaL} with $a\in \R$, $b>0$, and L\'evy measure $\nu$. A L\'evy process $\{L(t), \, t \geq 0\}$ such that $C \left\{ \zeta \ddagger L(1) \right\} = \kappa_L (\zeta) = \kappa(\zeta)$ is called the \textit{L\'evy seed} of $\Lambda$. In the more general context, $(a, b, \nu, Leb)$ is referred to as the \textit{characteristic quadruple} and $Leb$ as the \textit{control measure}. Note that to any infinitely divisible distribution there corresponds a homogeneous L\'evy basis on $\R^d$. The integration of deterministic functions with respect to the L\'evy basis can be defined first for real simple functions, then as a limit in probability of such integrals. More details can be found in \cite{rajput1989spectral}.

\subsection{Trawl processes}
Suppose $\Lambda$ is a homogeneous L\'evy basis on $\R^d \times \R$, $d \in \N$, with characteristic quadruple $(a, b, \nu, Leb)$ and let $\kappa=\kappa_L$ denote the cumulant function \eqref{e:kappaL} of the L\'evy seed process $L=\{L(t), \, t \geq 0\}$.

Let $A=A_0 \in \mathcal{B}_b(\R^d\times \R)$ be a Borel set of finite Lebesgue measure and for $t\in \R$ put $A_t = A + (\bm{0}, t)$. The \emph{trawl process} associated with L\'evy basis $\Lambda$ and \emph{trawl} $A$ is defined as
\begin{equation*}
  X(t) = \Lambda (A_t) = \int_{\R^d \times \R} \bm{1}_{A} (\bm{\xi}, s-t) \Lambda (d \bm{\xi} , ds), \quad t \in \R.
\end{equation*}
The process $\{X(t), \, t \in \R\}$ is strictly stationary (\cite{barndorff2014integer}) and
\begin{equation*}
  C \left\{ \zeta \ddagger X(t) \right\} = Leb(A) \kappa_L(\zeta)
\end{equation*}
The cumulants, if they exist, are given by
\begin{equation}\label{e:kappaXtoL}
  \kappa_X^{(m)} = Leb(A) \kappa_L^{(m)}
\end{equation}
where $\kappa_L^{(m)}$ denotes the $m$-th order cumulant of $L(1)$.

While specifying the infinitely divisible law of the L\'evy basis controls the marginal distribution of the trawl process, the choice of the trawl set $A$ determines the dependence structure of the process. For simplicity, we will assume in the following that $d=1$ so that $A \in \mathcal{B}_b(\R \times \R)$ and
\begin{equation}\label{Xdef}
  X(t) = \Lambda (A_t) = \int_{\R \times \R} \bm{1}_{A} (\xi, s-t) \Lambda (d \xi , ds), \quad t \in \R.
\end{equation}
The typical way to specify the trawl $A \in \mathcal{B}_b(\R \times \R)$ is to put
\begin{equation*}
  A = \left\{ (\xi, s) : 0 \leq \xi \leq g(-s), \ s \leq 0 \right\},
\end{equation*}
where $g : [0,\infty) \to [0,\infty)$ is a measurable function such that $Leb(A) < \infty$. Then, clearly
\begin{equation*}
  A_t = \left\{ (\xi, s) : 0 \leq \xi \leq g(t-s), \ s \leq t \right\}
\end{equation*}
and we can write
\begin{equation*}
   X(t) = \int_{\R \times (-\infty,t]} \bm{1}_{[0,g(t-s)]} (\xi) \Lambda (d \xi , ds), \quad t \in \R.
\end{equation*}
We will refer to $g$ as the \textit{trawl function} and in the following we always assume $g$ is nonincreasing and hence $g(-s)$, $s \in (-\infty,0]$ is nondecreasing.

By using \cite[Proposition 5.]{barndorff2015recent}, one can show that for $\zeta_1, \zeta_2 \in \R$ and $h\geq 0$
\begin{equation}\label{e:bivariate}
   \log \E e^{i \left( \zeta_1 X(0) + \zeta_2 X(h) \right) } = \int_{\R \times \R} \kappa_L \left( \zeta_1 \bm{1}_{A} (\xi, s) + \zeta_2 \bm{1}_{A} (\xi, s-h) \right) d \xi ds.
\end{equation}
Now if $E X(t)^2 < \infty$, then taking derivative with respect to $\zeta_1$ and $\zeta_2$ in \eqref{e:bivariate} and letting $\zeta_1, \zeta_2 \to 0$ we obtain
\begin{align*}
  \E X(t) X(t+h) = \int_{\R \times \R}  \bm{1}_{A} (\xi, s) \bm{1}_{A} (\xi, s-h)  d \xi ds = \int_{-\infty}^0 \int_0^{g(h-s)} d\xi ds = \int_{h}^\infty g(x) dx.
\end{align*}
Hence, the correlation function of the trawl process for $h\geq 0$ is
\begin{equation}\label{e:trawlcorr}
  r(h) = \corr \left( X(t), X(t+h) \right) = \frac{ \int_h^\infty g(x) d x}{ \int_0^\infty g(x) dx}.
\end{equation}
This shows how the choice of $g$ affects the dependence.

\begin{example}
Suppose for some $\alpha>0$, $g$ is $-(\alpha+1)$-regularly varying at infinity so that $g(x)=L(x) x^{-\alpha-1}$, with $L$ slowly varying at infinity, i.e.~for every $x>0$, $L(tx)/L(t)\to 1$ as $t\to \infty$. Then from \eqref{e:trawlcorr} by Karamata's theorem \cite[Proposition 1.5.10.]{bingham1989regular} we have that
\begin{equation*}
  r(h) \sim \frac{1}{\alpha \int_0^\infty g(x) dx} L(h) h^{-\alpha}, \quad \text{ as } h \to \infty.
\end{equation*}
In particular, by taking $\alpha \in (0,1)$ we can obtain a trawl process with non-integrable correlation function, a property well known as the \textit{long-range dependence}. The next example is a particular case.
\end{example}

\begin{example}\label{ex:gammatrawl}
Suppose $\{X(t), \, t \in \R\}$ is a trawl process with finite second moment specified by the trawl function
\begin{equation}\label{e:gammatrawlg}
  g(x) = (1+x)^{-\alpha-1},
\end{equation}
for some $\alpha>0$. From \eqref{e:trawlcorr} it follows that the correlation function is
\begin{equation*}
  r(h) = (1+h)^{-\alpha}, \quad h \geq 0.
\end{equation*}
In \cite{barndorff2014integer}, the same example is obtained indirectly as a special case of the so-called superposition trawl. The general superposition trawl is specified by the trawl function
\begin{equation*}
  \widetilde{g}(x) = \int_0^\infty e^{-\lambda x} \pi(d\lambda), \qquad x \geq 0,
\end{equation*}
where $\pi$ is some probability measure on $(0,\infty)$ such that $\int_0^\infty \lambda^{-1} \pi(d\lambda) < \infty$. Taking $\pi$ to be the Gamma distribution $\Gamma(1+\alpha,1)$ distribution, defined by the density
\begin{equation*}
  f(\lambda)= \frac{1}{\Gamma(1+\alpha)} \lambda^{\alpha} e^{-\lambda} \bm{1}_{(0,\infty)}(\lambda),
\end{equation*}
we obtain a trawl specified by \eqref{e:gammatrawlg}. Such a modelling framework is motivated by the similar approach used in superpositions of Ornstein-Uhlenbeck type processes (see \cite{barndorff2001superposition}).
\end{example}

\subsection{Integrated process}
Given a trawl process $\{X(t), \ t \in \R\}$ we will denote by $\{X^*(t), \ t \geq 0\}$ the integrated process
\begin{equation}\label{X*def}
  X^*(t) = \int_0^t X(u) du.
\end{equation}
The following lemma expresses cumulants of the integrated process $\kappa_{X^*}^{(m)} (t)$ in terms of the cumulants $\kappa_L^{(m)}$ of the L\'evy seed. We will assume that the cumulant function $\kappa_L$ of the L\'evy seed process is analytic in a neighborhood of the origin. A sufficient condition for the analyticity of $\kappa_L$ in the neighborhood of the origin is that there exists $a>0$ such that $\E e^{a |L(1)|} <\infty$ \cite[Theorem 7.2.1]{lukacs1970characteristic}. This implies in particular that all the moments and cumulants of $X(t)$ exist. Many infinitely divisible distributions satisfy this condition, for example, inverse Gaussian, normal inverse Gaussian, gamma, variance gamma, tempered stable (see \cite{GLST2016} for details).

\begin{lemma}\label{lemma:cumulants}
Suppose $\{X(t), \ t \in \R\}$ is a trawl process \eqref{Xdef} such that the cumulant function $\kappa_L$ of the L\'evy seed process $\{L(t)\}$ is analytic in a neighborhood of the origin. The cumulants of $X^*(t)$ are then given by
\begin{equation}\label{e:kappaX*}
  \kappa_{X^*}^{(m)} (t) = \kappa_L^{(m)} \int_{\R \times \R} \left( h_A(\xi, s, t) \right)^m d \xi ds, \quad m \geq 1,
\end{equation}
where $\kappa_L^{(m)}$ is the $m$-th order cumulant of the L\'evy seed process $L$ and
\begin{equation}\label{e:hA}
  h_A(\xi, s, t) = \int_0^t  \bm{1}_{A} (\xi, s-u) du = \int_0^t  \bm{1}_{(-\infty, g(u-s)]} (\xi) \bm{1}_{(-\infty,u]}(s) du.
\end{equation}
\end{lemma}

\begin{proof}
From \cite[Proposition 5.]{barndorff2015recent} it follows that
\begin{equation*}
  C \left\{ \zeta \ddagger X^*(t) \right\} = C \left\{ \zeta \ddagger \int_0^t X(u) du \right\} = \int_{\R \times \R} C \left\{ \zeta h_A(\xi, s, t) \ddagger L(1) \right\} d \xi ds.
\end{equation*}
with $h_A(\xi, s, t)$ given by \eqref{e:hA}. By the analyticity of $\kappa_L$ we have
\begin{equation*}
  C \left\{ \zeta \ddagger L(1) \right\} = \sum_{m=1}^\infty \kappa_L^{(m)} \frac{\left(i \zeta\right)^{m}}{m!}
\end{equation*}
and so
\begin{align*}
  C \left\{ \zeta \ddagger X^*(t) \right\} &= \int_{\R \times \R} \sum_{m=1}^\infty \kappa_L^{(m)} \frac{\left(i \zeta\right)^{m}}{m!} \left( h_A(\xi, s, t) \right)^m d \xi ds\\
  &= \sum_{m=1}^\infty \left( \kappa_L^{(m)} \int_{\R \times \R} \left( h_A(\xi, s, t) \right)^m d \xi ds \right)  \frac{\left(i \zeta\right)^{m}}{m!}.
\end{align*}
\end{proof}

\section{Intermittency}\label{sec3}
Intermittency is a property used to describe models exhibiting sharp fluctuations in time and a high degree of variability. The term has a precise definition in the theory of stochastic partial differential equations, where it is characterized by the Lyapunov exponents (see e.g.~\cite{zel1987intermittency,carmona1994parabolic,khoshnevisan2014analysis,chen2015moments}).

Here, we follow \cite{GLST2016} and define intermittency as a property which indicates that the stochastic process does not have a typical limiting behavior of moments. Intermittency is characterized by the scaling function. The \textit{scaling function} of the process $Y=\{Y(t),\, t \geq 0\}$ is defined in the range of finite moments $(0,\overline{q}(Y))$, $\overline{q}(Y) = \sup \{ q >0 :\E|Y(t)|^q < \infty  \ \forall t\}$ as the limit
\begin{equation}\label{deftau}
    \tau_Y(q) = \lim_{t\to \infty} \frac{\log \E |Y(t)|^q}{\log t},
\end{equation}
assuming the limit exists and is finite. It can be shown that $\tau_Y$ is always convex and $q \mapsto \tau_Y(q)/q$ is non-decreasing (\cite{GLST2015}).

\begin{defi}
A stochastic process $Y=\{Y(t),\, t \geq 0\}$ is \textit{intermittent} if there exist some $p, r \in (0,\overline{q}(Y))$ such that
\begin{equation}\label{intermittency}
    \frac{\tau_Y(p)}{p} < \frac{\tau_Y(r)}{r},
\end{equation}
that is, $\tau_Y(q)/q$ is strictly increasing at some $q$.
\end{defi}

Recall that the process $Y$ is $H$-self-similar if for any $c>0$, $\{Y(ct)\}\overset{d}{=} \{c^H Y(t)\}$, where $\{\cdot\} \overset{d}{=} \{\cdot\}$ denotes the equality of finite dimensional distributions. If $Y$ is a $H$-self-similar process, then $\tau_Y(q)=Hq$, and $\tau_Y(q)/q$ is constant, therefore the process is not intermittent. Recall that by Lamperti's theorem (see, for example, \cite[Theorem 2.1.1]{embrechts2002selfsimilar}), if as $n \to \infty$
\begin{equation}
\left\{ \frac{Y(nt)}{A_n} \right\} \overset{d}{\to} \left\{ Z(t) \right\}, \label{limitform}
\end{equation}
where $\{\cdot\} \overset{d}{\to} \{\cdot\}$ means convergence of all finite-dimensional distributions, $Z(t)$ is always a self-similar process and the normalizing sequence must be of the form $A_n=L(n) n^H$ for some $H>0$ and $L$ slowly varying at infinity. From here, one can show that as soon as \eqref{limitform} holds, then there is $H>0$ such that for every $q>0$ satisfying
\begin{equation}\label{limitformmom}
    \frac{\E| Y(nt)|^q}{A_n^q} \to \E |Z(t)|^q, \quad \forall t \geq 0,
\end{equation}
one has that $\tau_Y(q)= H q$. In this setting, $Y$ usually represents some form of cumulative process, e.g.~partial sum process or integrated process. Hence, when intermittency is present, \eqref{limitform} and \eqref{limitformmom} cannot both hold (see \cite{GLST2016} for details).

The following theorem establishes intermittency of certain integrated trawl process. For the L\'evy seed, any infinitely divisible distribution is allowed provided it has cumulant function analytic in the neighbourhood of the origin. However, the Gaussian case is excluded. In the Gaussian case one can apply \cite[Lemma 5.1]{Taqqu1975} and obtain limit theorems with convergence to fractional Brownian motion (see \cite[Example 9]{GLST2016} for the similar argument). The underlying trawl process is assumed to a trawl function regularly varying at infinity. Additionally, the trawl function is assumed to be continuously differentiable and decreasing. An example of such trawl is given in Example \ref{ex:gammatrawl}.

\begin{theorem}\label{thm:interm}
Let $\{X(t), \, t \in \R\}$ be a zero mean non-Gaussian trawl process such that the cumulant function $\kappa_L$ of the L\'evy seed process is analytic in the neighborhood of the origin and suppose the trawl function $g$ is continuously differentiable, decreasing  and $(-\alpha-1)$-regularly varying at infinity for some $\alpha >0$. If $\tau_{X^*}$ is the scaling function \eqref{deftau} of the process $X^*=\{X^*(t), \, t \geq 0\}$ in \eqref{X*def}, then for every $q\geq q^*$
\begin{equation*}
\tau_{X^*}(q) = q-\alpha,
\end{equation*}
where $q^*$ is the smallest even integer greater than $2\alpha$. In particular, for $q^*\leq p < r$
\begin{equation*}
\frac{\tau_{X^*}(p)}{p} < \frac{\tau_{X^*}(r)}{r}
\end{equation*}
and hence $X^*$ is intermittent.
\end{theorem}

\begin{proof}
First, we will investigate the asymptotic behavior of $\kappa_{X^*}^{(m)}(t)$ for $m \in \N$ as $t\to \infty$ using \eqref{e:kappaX*}. By the assumptions, the trawl function $g:[0,\infty)\to (0,g(0)]$ is invertible and we can rewrite \eqref{e:hA} in the following form
\begin{align*}
  h_A(\xi,s,t) &= \int_0^t   \bm{1}_{[0,g(u-s)]} (\xi) \bm{1}_{(-\infty,u]}(s) du\\
  &= \int_0^t   \bm{1}_{(-\infty, g^{-1}(\xi) + s]} (u) \bm{1}_{[s,\infty)}(u) du.
\end{align*}
From here we conclude that $h_A(\xi,s,t)=0$ if either $s>t$ or $\xi<0$ or $\xi>g(0)$ or $g^{-1}(\xi)<-s$ (which is equivalent to $\xi>g(-s)$ for $s\leq0$). Otherwise, we have for $s \leq 0$
\begin{equation*}
  h_A(\xi,s,t) = \int_0^t \bm{1}_{[0, g^{-1}(\xi) + s]} (u) du = \left( g^{-1}(\xi) + s \right) \wedge t
\end{equation*}
and for $s>0$
\begin{equation*}
  h_A(\xi,s,t) = \int_0^t \bm{1}_{[s, g^{-1}(\xi) + s]} (u) = \left( \left( g^{-1}(\xi) + s \right) \wedge t \right) - s.
\end{equation*}
Hence, we can write
\begin{equation}\label{e:hA:fullform}
  h_A(\xi, s, t) = \begin{cases}
  t, & \text{ if } s \leq 0 \text{ and } 0 \leq \xi \leq g(t-s),\\
  g^{-1}(\xi) + s, & \text{ if } s \leq 0 \text{ and } g(t-s) < \xi \leq g(-s),\\
  t-s, & \text{ if } 0 < s \leq t \text{ and } 0 \leq \xi \leq g(t-s),\\
  g^{-1}(\xi), & \text{ if } 0 < s \leq t \text{ and } g(t-s) < \xi \leq g(0),\\
  0, & \text{ otherwise}.
  \end{cases}
\end{equation}
The cumulants of the integrated process \eqref{e:kappaX*} can now be expressed as
\begin{equation}\label{e:proof:kappaX*}
  \kappa_{X^*}^{(m)} (t) = \kappa_L^{(m)} \left( I_1^{(m)}(t) + I_2^{(m)}(t) + I_3^{(m)}(t) + I_4^{(m)}(t) \right),
\end{equation}
where
\begin{align*}
  I_1^{(m)}(t) &= \int_{-\infty}^0 \int_0^{g(t-s)} t^m d \xi ds,\\
  I_2^{(m)}(t) &= \int_{-\infty}^0 \int_{g(t-s)}^{g(-s)} (g^{-1}(\xi) + s)^m d \xi ds,\\
  I_3^{(m)}(t) &= \int_0^t \int_0^{g(t-s)} (t-s)^m d \xi ds,\\
  I_4^{(m)}(t) &= \int_0^t \int_{g(t-s)}^{g(0)} (g^{-1}(\xi))^m d \xi ds.
\end{align*}
We assumed $X$ has zero mean, hence $\kappa_L^{(1)}=0$ from \eqref{e:kappaXtoL} and consequently $\kappa_{X^*}^{(1)}=0$.

\smallskip
\textbf{Case $m> \alpha +1$}. We now take $m> \alpha +1$ and consider each integral one by one. Since $g$ is $(-\alpha-1)$-regularly varying, it can be written in the form $g(x)=L(x) x^{-\alpha-1}$ with $L$ slowly varying at infinity. For $I_1^{(m)}(t)$ using a change of variable and Karamata's theorem \cite[Proposition 1.5.10.]{bingham1989regular} we get
\begin{equation}\label{proof:I1m}
  I_1^{(m)}(t) = t^m \int_{-\infty}^0 g(t-s) ds = t^m \int_t^\infty g(u) du \sim \frac{1}{\alpha} L(t) t^{m-\alpha}, \quad \text{ as } t \to \infty.
\end{equation}
For the second integral, since $g$ is assumed to be continuously differentiable, we have by the change of variable $u=g^{-1}(\xi)+s$ and Fubini's theorem
\begin{equation*}
  I_2^{(m)}(t) =\int_{-\infty}^0 \int_{t}^{0} u^m g'(u-s) du ds = \int_{t}^{0} u^m  \int_{-\infty}^0  g'(u-s) ds du = \int_{0}^{t} u^m g(u) du.
\end{equation*}
Now from \cite[Proposition 1.5.11. (i)]{bingham1989regular} it follows that
\begin{equation*}
  I_2^{(m)}(t) \sim \frac{1}{m-\alpha} L(t) t^{m-\alpha}, \quad \text{ as } t \to \infty.
\end{equation*}
Similarly, for $I_3^{(m)}(t)$ we obtain
\begin{equation*}
  I_3^{(m)}(t) = \int_0^t (t-s)^m g(t-s) ds = \int_0^t u^m g(u) du  \sim \frac{1}{m-\alpha} L(t) t^{m-\alpha}, \quad \text{ as } t \to \infty.
\end{equation*}
Finally, for $I_4^{(m)}(t)$ by the change of variable $u=g^{-1}(\xi)$, Fubini's theorem and integration by parts it follows that
\begin{align}
  I_4^{(m)}(t) &= \int_0^t \int_{t-s}^{0} u^m g'(u) du ds\nonumber\\
  &= - \int_0^t u^m g'(u) \int_0^{t-u} ds du\nonumber\\
  &= \int_0^t u^{m+1} g'(u) du - t \int_0^t u^m g'(u) du\nonumber\\
  &= t^{m+1} g(t) - (m+1) \int_0^t u^{m} g(u) du - t^{m+1} g(t) + t m \int_0^t u^{m-1} g(u) du\nonumber\\
  &= t m \int_0^t u^{m-1} g(u) du - (m+1) \int_0^t u^{m} g(u) du.\label{e:I4m}
\end{align}
Since we have assumed $m>\alpha+1$ and \cite[Proposition 1.5.11. (i)]{bingham1989regular} can be applied to get as $t \to \infty$
\begin{equation*}
  I_4^{(m)}(t) \sim  \frac{m}{m-\alpha-1} L(t) t^{m-\alpha} - \frac{m+1}{m-\alpha} L(t) t^{m-\alpha} = \frac{\alpha+1}{(m-\alpha-1)(m-\alpha)} L(t) t^{m-\alpha}.
\end{equation*}

We now conclude from \eqref{e:proof:kappaX*} that for every $m>\alpha+1$ such that $\kappa_L^{(m)}\neq 0$ there exists a slowly varying function $L_m$ such that $\kappa_{X^*}^{(m)} (t) \sim L_m(t) t^{m-\alpha}$.

\textbf{Case $m< \alpha +1$}. In this case we will only need an upper bound on $\kappa_{X^*}^{(m)}(t)$. The equation \eqref{proof:I1m} remains valid anyway and shows that $I_1^{(m)}(t)\leq C_1 t$ for $t$ large enough. Next, since $u^m g(u)=u^{m-\alpha-1}L(u)$ is bounded at infinity, we have
\begin{equation*}
\left| I_2^{(m)}(t) \right| = \left| I_3^{(m)}(t) \right| = \int_0^t u^m g(u) du \leq C_2 t.
\end{equation*}
Similarly, we can take $0<\varepsilon<\alpha+1-m$ and $u$ large enough so that $u^{m-1} g(u)=u^{m-\alpha-2}L(u)\leq C_3 u^{-1-\varepsilon}$. Hence, we have from \eqref{e:I4m}
\begin{equation*}
\left| I_4^{(m)}(t) \right| \leq t m \int_0^t u^{-1-\varepsilon} du + (m+1) C_2 t \leq C_4 t.
\end{equation*}
We conclude from \eqref{e:proof:kappaX*} that $\left| \kappa_{X^*}^{(m)}(t) \right| \leq C t$ for $m<\alpha+1$.

\smallskip
\textbf{Case $m= \alpha +1$}. Note that this is possible only if $\alpha$ is an integer. If the slowly varying function $L$ is bounded, everything remains the same as in proof of the previous case. Otherwise, for arbitrary $\varepsilon>0$, we can take $u$ large enough so that $L(u)\leq u^\varepsilon$. Now one can proceed as in the previous case to obtain that $\left| \kappa_{X^*}^{(m)}(t) \right| \leq C t^{1+\varepsilon}$ for $m=\alpha+1$.

Having established these results now, we can relate cumulants to moments as in the proof of \cite[Theorem 7]{GLST2016} and show that for some slowly varying function $\widetilde{L}$
\begin{equation*}\label{proof:conclusion1}
E|X^*(t)|^m \sim \widetilde{L}(t) t^{m-\alpha}
\end{equation*}
and consequently $\tau_{X^*}(m) = m-\alpha$, for any even integer $m$ greater than $2 \alpha$. As in \cite[Lemma 3]{GLST2016}, the convexity of $\tau_{X^*}$ is then used to extend the validity of $\tau_{X^*}(q) = q-\alpha$ to any real $q\geq q^*$.
\end{proof}

\bigskip

\textbf{Acknowledgments:} Nikolai N. Leonenko was supported in part by projects MTM2012-32674 (co-funded by European Regional Development Funds), and MTM2015--71839--P, MINECO, Spain. This research was also supported under Australian Research Council's Discovery Projects funding scheme (project number DP160101366), and under Cardiff Incoming Visiting Fellowship Scheme and International Collaboration Seedcorn Fund.

Murad S.~Taqqu was supported by the NSF grant DMS-1309009 at Boston University.

\bibliographystyle{agsm}
\bibliography{References}

\end{document}